\documentclass{amsart}
\usepackage{booktabs} 
\usepackage{diagbox}
\newtheorem{theorem}{Theorem}[section]
\newtheorem{lemma}[theorem]{Lemma}
\newtheorem{corollary}[theorem]{Corollary}

\theoremstyle{definition}

\theoremstyle{remark}

\numberwithin{equation}{section}

\begin{document}

\title{The Application of Tridiagonal Matrices in P-polynomial Table Algebras}

\author{Masoumeh Koohestani}
\address{Faculty of Mathematics, K. N. Toosi
University of Technology, P. O. Box $16315$-$1618$, Tehran, Iran}

\email{m.kuhestani@email.kntu.ac.ir}

\author{Amir Rahnamai Barghi}
\address{Faculty of Mathematics, K. N. Toosi
University of Technology, P. O. Box $16315$-$1618$, Tehran, Iran}
\email{rahnama@kntu.ac.ir}
\author{Amirhossein Amiraslani}
\address{STEM Department, University of Hawaii-Maui College, Kahului, HI 96732, USA
}
\email{aamirasl@hawaii.edu}






\begin{abstract}
In this paper, we study the characters of two classes of P-polynomial table algebras using tridiagonal matrices. To this end, we obtain some results about the eigenstructure of special tridiagonal matrices. We also find a recursion relation for the characteristic polynomial of the first intersection matrix of P-polynomial table algebras by means of LU factorization.
\end{abstract}
\maketitle
\section{Introduction}
The eigenstructure of tridiagonal matrices and their applications have been studied exhaustively in many papers such as \cite{eigensystem}, \cite{class},  \cite{powers}, \cite{network} and \cite{certain}. Tridiagonal matrices are also used in P-polynomial table algebras. More precisely, the first intersection matrix of a P-polynomial table algebra is
a tridiagonal matrix whose eigenvalues can give all characters of the P-polynomial table algebra, see {\cite[Remark $3.1$]{table}}.
The  study of characters of table algebras is important and can be used in studying  the properties of association schemes, because the Bose-Mesner algebra of any association scheme is a table algebra, see \cite{characters}. However, calculating the characters of table algebras explicitly is sometimes hard or impossible.

Here, we intend to develop some linear algebra methods for tridiagonal matrices which can be used in calculating the characters of P-polynomial table algebras.  Hence, we first consider two matrices, namely $A_n$ and $P_n$, in the forms of
\begin{equation}\label{mat}
A_n= \left(
\begin{array}{cccccc}
0&1&&&&\\
2&0&1&&&\\
&1&0&1&&\\
&&1&0 & \ddots &\\
&&&\ddots &\ddots &1\\
&&&&1&0
\end{array} \right)_{n\times n},
\end{equation}
\begin{equation}\label{mat0}
P_n= \left(
\begin{array}{cccccc}
0&1&&&&\\
a&0&1&&&\\
&a&0&1&&\\
&&a&0& \ddots &\\
&&&\ddots &\ddots &1\\
&&&&a&0
\end{array} \right)_{n\times n},
\end{equation}
where $0\neq a\in \mathbb{C}$ .
The eigenvalues of $A_n$ and $P_n$ may be obtained from the results in \cite{eigensystem} and \cite{several}, but we need their characteristic polynomials to study the characters of some P-polynomial table algebras. To this end, we calculate the characteristic polynomials and the eigenvalues of $A_n$ and $P_n$ through an approach which is used in the other parts of this paper as well.
We then apply our results for $A_n$ to study the eigenvalues of a special class of tridiagonal matrices, $Q_n$, which is
\begin{equation}\label{mat1}
Q_n= \left(
\begin{array}{cccccc}
a&b&&&&\\
2c&a&b&&&\\
&c&a&b&&\\
&&c&a & \ddots &\\
&&&\ddots &\ddots &b\\
&&&&c&a
\end{array} \right)_{n\times n},
\end{equation}
where $a, b, c\in \mathbb{C}$ and $bc\neq 0$.

We also use our results to study the characters of  two classes of P-polynomial table algebras. The first intersection  matrices of these table algebras have the following forms:
\begin{equation}\label{mat2}
 \left(
\begin{array}{cccccc}
0&1&&&&\\
2\alpha ^2&0&\alpha &&&\\
&\alpha &0&\alpha &&\\
&&\ddots &\ddots & \ddots &\\
&&&\alpha &0 &\alpha\\
&&&&\alpha &\alpha
\end{array} \right)_{(d+1)\times (d+1)},
\end{equation}
and
\begin{equation}\label{mat3}
 \left(
\begin{array}{cccccc}
0&1&&&&\\
2\alpha \gamma &0&\gamma &&&\\
&\alpha &0&\gamma &&\\
&&\ddots &\ddots & \ddots &\\
&&&\alpha &0 &\gamma \\
&&&&2\alpha &0
\end{array} \right)_{(d+1)\times (d+1)},
\end{equation}
where $\alpha, \gamma \in\mathbb{R}^+$.
Note that the above table algebras are of dimension $d$ and introduced in \cite{perfect}.

Finally,
 we propose an LU factoring technique with special pivoting for a generic tridiagonal matrix. Using the results, we explore and discuss the eigenstructure of the intersection matrix of  P-polynomial table algebras.

 Throughout this paper, we  denote the complex numbers
 and the positive real numbers  by $\mathbb{C}$ and $\mathbb{R}^+$, respectively.
 
 \section{P-polynomial Table Algebras}\label{sec}
\label{tablealgebra}
In this section, we review some important concepts from table algebras and P-polynomial table algebras; see~\cite{table} and \cite{scheme} for more details.

Let $A$ be an associative commutative algebra with finite-dimension and a basis $\mathbf{B}=\{x_0=1_A, x_1, \cdots , x_d\}$. Then $(A, \mathbf{B})$ is called a table algebra if the following conditions hold:
\begin{itemize}
\item[(i)]
$\displaystyle{x_ix_j=\sum_{m=0}^d\beta _{ijm}x_m}$ with $\beta _{ijm} \in \mathbb{R}^+\cup \{0\}$, for all $i$, $j$;
\item[(ii)]
there is an algebra automorphism of $A$ (denoted by $^-$),
whose order divides 2, such that if ${x}_i\in \mathbf{B}$, then $\overline{x}_i\in \mathbf{B}$ and $\overline{i}$ is defined by $x_{\overline{i}}=\overline{x}_i$;
\item[(iii)]
for all $i$, $j$, we have $\beta_{ij0}\neq 0$ if and only if $j=\overline{i}$; moreover, $\beta_{i\overline{i}0}>0$.
\end{itemize}
Let $(A,\mathbf{B})$ with $\mathbf{B}=\{x_0=1_A, x_1, \cdots , x_d\}$ be a table algebra. Then $(A,\mathbf{B})$
is called real, if $i=\overline{i}$, for $0 \leq i\leq d$.
The $i$-th intersection matrix of $(A,\mathbf{B})$ (the intersection matrix with respect to $x_i$) is the
 matrix of the form:
$$
B_i=\left(
\begin{array}{cccc}
\beta _{i00} & \beta _{i01} & \ldots & \beta _{i0d}\\
\beta _{i10} & \beta _{i11} & \ldots & \beta _{i1d}\\
\vdots & \vdots & \ddots & \vdots \\
\beta _{id0} & \beta _{id1} & \ldots & \beta _{idd}\\
\end{array} \right)_{(d+1) \times (d+1)},
$$
where $x_ix_j=\displaystyle{\sum _{i=0}^d\beta _{ijk}x_k}$, for all $i, j, k$.

For any table algebra $(A,\mathbf{B})$ with $\mathbf{B}=\{x_0=1_A, x_1, \cdots , x_d\}$, there exists a unique algebra homomorphism $f:A\rightarrow \mathbb{C}$ such that $f(x_i)=f(x_{\overline{i}})\in \mathbb{R}^+$, for $0 \leq i\leq d$, see \cite{scheme}. If $f(x_i)=\beta_{i\overline{i}0}$ for all $i$, then $(A,\mathbf{B})$ is called standard.

 A real standard table algebra $(A,\mathbf{B})$ with $\mathbf{B}=\{x_0=1_A, x_1, \cdots , x_d\}$ is called P-polynomial if for each $i$, $2\leq i\leq d$, there exists a complex cofficient polynomial $\nu_i(x)$ of degree $i$ such that $x_i=\nu_i(x_1)$. If $(A,\mathbf{B})$ is a P-polynomial table algebra, then for all $i$, there exist $b_{i-1}, a_i, c_{i+1}\in \mathbb{R}$ such that
\begin{equation}\label{abc}
x_1x_i=b_{i-1}x_{i-1}+a_ix_i+c_{i+1}x_{i+1},
\end{equation}
with $b_{i}\neq 0$, ($0\leq i\leq d-1$), $c_{i}\neq 0$, ($1\leq i\leq d$), and $b_{-1}=c_{d+1}=0$. Hence, the first intersection matrix of a P-polynomial table algebra is as follows.
\begin{equation}\label{matrix}
B_1=\left(
\begin{array}{ccccc}
a_0&c_1&&&\\
b_0&a_1&c_2&&\\
&b_1&a_2&\ddots&\\
&&\ddots &\ddots &c_d\\
&&&b_{d-1}&a_d
\end{array} \right)_{(d+1) \times (d+1)}.
\end{equation}
Let $(A,\mathbf{B})$ with $\mathbf{B}=\{x_0=1_A, x_1, \cdots , x_d\}$ be a table algebra. Since $A$ is semisimple, the primitive idempotents of $A$ form another basis for $A$, see \cite{scheme}.
Consequently, if $\{e_0, e_1, \cdots, e_d\}$ is the set of the primitive idempotents of $A$, then we have $x_i=\sum_{j=0}^dp_i(j)e_j$, where $p_i(j)\in \mathbb{C}$, for $0\leq i, j\leq d$. The numbers $p_i(j)$ are the characters of the table algebra.

Let $(A,\mathbf{B})$ with $\mathbf{B}=\{x_0=1_A, x_1, \cdots , x_d\}$ be a P-polynomial table algebra. Then the $p_1(j)$ are equal to the eigenvalues of its first intersection matrix and for $2\leq i\leq d$, we have
\begin{equation}\label{character}
p_i(j)=\nu _i(p_1(j)),
\end{equation}
where $\nu _i(x)$ is a complex cofficient polynomial such that $x_i=\nu _i(x_1).$

\section{Tridiagonal Matrices} \label{sec1}

We now calculate the eigenvalues of $A_n$ which are given in (\ref{mat}).
The following lemma helps us to calculate the eigenvalues of tridiagonal matrices.
\begin{lemma}\label{lem1}
(\cite{Fib}) If $\{H_n, n = 1,2,\cdots \}$ is a sequence of tridiagonal matrices of the form
$$
H_n= \left(
\begin{array}{ccccc}
h_{1,1}&h_{1,2}&&&\\
h_{2,1}&h_{2,2}&h_{2,3}&0&\\
&h_{3,2}&h_{3,3}&\ddots &\\
&0&\ddots & \ddots &h_{n-1,n}\\
&& &h_{n,n-1} &h_{n,n}\\
\end{array} \right),
$$
then the determinants of $H_n$ are given by the recursive formula:
\begin{align}
|H_1|&= h_{1,1},\nonumber \\
|H_2|&=h_{1,1}h_{2,2}-h_{1,2}h_{2,1},\nonumber\\
|H_n|&=h_{n,n}|H_{n-1}|-h_{n-1,n}h_{n,n-1}|H_{n-2}|.\nonumber
\end{align}
\end{lemma}
\begin{theorem} \label{theorem1}
The eigenvalues of $A_n$, as given in (\ref{mat}), are as follows.
$$
x_k=2\cos\left( \frac{(2k+1)\pi}{2n}\right),\,\,\ k=0, 1, \cdots , n-1.
$$
\end{theorem}
\begin{proof}
Let $D_n(x)=|xI_n-A_n|$. From Lemma \ref{lem1}, it is concluded that
\begin{equation}\label{mat33}
D_n(x)=
\left|
\begin{array}{cccccc}
x&1&&&&\\
2&x&1&&&\\
&1&x&1&&\\
&&1&x& \ddots &\\
&&&\ddots &\ddots &1\\
&&&&1&x
\end{array}\right|.
\end{equation}
Now, we define the function $\Delta _n(x)$  as
\begin{align*}
\Delta _n(x)=& \left|
\begin{array}{cccccc}
x&1&&&&\\
1&x&1&&&\\
&1&x&1&&\\
&&1&x& \ddots &\\
&&&\ddots &\ddots &1\\
&&&&1&x
\end{array} \right| _{n\times n}.\end{align*}
If we calculate the determinant given in (\ref{mat33}) using the Laplace
expansion, then we have
\begin{equation}\label{eq6}
D_n(x)=x\Delta _{n-1}(x) -2 \Delta _{n-2}(x).
\end{equation}
On the other hand, from Lemma \ref{lem1}, we get
\begin{align}\label{cheb}
\Delta _n(x)=&~x\Delta _{n-1}(x)-\Delta _{n-2}(x),~~ n\geq 3,
\end{align}
with $\Delta _2(x)=x^2-1$ and $\Delta _1(x)=x$.
Solving (\ref{cheb}), we obtain
\begin{equation}\label{eq7}
\Delta _n(x)= U_n(\frac{x}{2}),
\end{equation}
where $U_n(x)$ is the n-th degree Chebyshev polynomial of  second kind, see \cite{oxford} for more details.
Next, from (\ref{eq6}) , (\ref{cheb}) and (\ref{eq7}), we have
\begin{align}\label{value}
D_n(x)&=~xU_{n-1}(\frac{x}{2})-2U_{n-2}(\frac{x}{2})\nonumber\\
&=~U_{n}(\frac{x}{2})-U_{n-2}(\frac{x}{2})\nonumber\\
&=~2T_n(\frac{x}{2}),\nonumber
\end{align}
where $T_n(x)$ is the $n$-th Chebyshev polynomial of first kind, for more details about Chebyshev polynomials, see \cite{oxford}.
Hence, the eigenvalues of $A_n$ are
$$x_k=2\cos\left( \frac{(2k+1)\pi}{2n}\right),\,\,\ k=0, 1, \cdots , n-1.$$
\end{proof}
We now calculate the characteristic polynomial of $P_n$, given in (\ref{mat0}), in the following lemma.
\begin{lemma}\label{lem2}
The characteristic polynomial of $P_n$, given in (\ref{mat0}), is $\left (\sqrt{a}\right)^n U_n\left( \frac{x}{2\sqrt{a}}\right)$.
\end{lemma}
\begin{proof}
Let $H_n(x)=|xI_n-P_n|$ be the characteristic polynomial of $P_n$. Through some straightforward calculations, we obtain that
$$H_1(x)=\sqrt{a}~\Delta _1\left(\frac{x}{\sqrt{a}}\right),~~H_2(x)=a~\Delta _2\left(\frac{x}{\sqrt{a}}\right),$$
where $\Delta_n(x)$ is as given in Theorem \ref{theorem1}. We now apply Lemma \ref{lem1} along with an inductive argument to prove that $H_n(x)=(\sqrt{a})^n\Delta _n\left(\frac{x}{\sqrt{a}}\right)$.
\begin{align*}
 H_n(x)=~&xH_{n-1}(x) -aH_{n-2}(x) \nonumber\\
 =~&x(\sqrt{a})^{n-1}\Delta _{n-1}\left(\frac{x}{\sqrt{a}}\right)-(\sqrt{a})^n\Delta _{n-2}\left(\frac{x}{\sqrt{a}}\right)\nonumber\\
 =~&(\sqrt{a})^n\Delta _n\left(\frac{x}{\sqrt{a}}\right),
\end{align*}
and the proof is complete.
\end{proof}
We now generalize $A_n$ to $Q_n$, given in (\ref{mat1}), and calculate its eigenvalues.
\begin{corollary}\label{cor}
The eigenvalues of $Q_n$ which is given in (\ref{mat1}) are
$$x _k=a+2\sqrt{bc}\,\cos\left( \frac{(2k+1)\pi}{2n}\right),\,\,\ k=0, 1, \cdots , n-1.$$
\begin{proof}
Let $C_n(x)=|xI_n-Q_n|$.
One can prove using induction on $n$ and Lemma \ref{lem1} that
\begin{equation}\label{I}
C_n(x) = (\sqrt{bc})^{n} D_{n}\left(\frac{x-a}{\sqrt{bc}}\right), ~~~ n \geq 1,
\end{equation}
where $D_n(x)$ is
given in Theorem \ref{theorem1}.
Then from Theorem \ref{theorem1}, the eigenvalues of $Q_n$ are

$$x _k=a+2\sqrt{bc}\,\cos\left( \frac{(2k+1)\pi}{2n}\right),\,\,\ k=0, 1, \cdots , n-1.$$
\end{proof}
\end{corollary}

\section{Two Classes of P-polynomial Table Algebras}\label{sec4}

 We now apply the results  and methods in Section \ref{sec1} to study the characters of two classes of
 P-polynomial table algebras. These table algebras are studied in \cite{perfect} and their first intersection matrices are given in (\ref{mat2}) and (\ref{mat3}).

\begin{theorem}\label{thm1}
Let $(A, \mathbf{B})$ be a P-polynomial table algebra
with $\mathbf{B}=\{x_0=1_A, x_1, \cdots , x_d\}$ and the first intersection matrix $B_1$ which is as given in (\ref{mat2}).
Then the characters of $(A, \mathbf{B})$ are
$$
\begin{array}{l}
p_0(j)=1,\\
\\
p_1(j)=x_j=2\alpha \cos\left( \frac{2k\pi}{2d+1}\right),\\
\\
p_i(j)=\left(\sqrt{\alpha}\right)^{i-4}\Bigg[(x_j^2-2\alpha ^2)U_{i-2}\left(\frac{x_j}{2\sqrt{\alpha}}\right)-
\alpha \sqrt{\alpha}x _jU_{i-3}\left(\frac{x_j}{2\sqrt{\alpha}}\right)\Bigg], ~~~ 2\leq i\leq d,
\end{array}
$$
for $0\leq j \leq d$.
\begin{proof}
For each $i$, $0\leq i\leq d$, the $p _i(j)$, $0\leq j \leq d$, are equal to the eigenvalues of the $i$-th intersection matrix $B_i$.
Since $B_0=I_{d+1}$, we have $p _0(j) = 1$ for all $j$. Similarly, the $p _1(j)$ are equal to the eigenvalues of $B_1$.
Set $R_{d+1}(x)=|xI_{d+1}-B_1|$. Let $N_n$ be a tridiagonal matrix of the form
$$\left(
\begin{array}{cccccc}
0&\alpha &&&&\\
2\alpha &0&\alpha &&&\\
&\alpha &0&\alpha &&\\
&&\ddots &\ddots & \ddots &\\
&&&\alpha &0 &\alpha\\
&&&&\alpha &\alpha
\end{array} \right)_{n\times n},$$
and $K_n(x)=|xI_{n}-N_n|$. From Lemma \ref{lem1}, we can conclude that $R_{d+1}(x)=K_{d+1}(x)$. Thus, we find the eigenvalues of $N_n$ instead of $B_1$. By Laplace expansion, we get
$$K_{d+1}(x)=(x-\alpha)C_{d}(x)-\alpha ^2 C_{d-1}(x),$$
where $C_n(x)$ is the characteristic polynomial of $Q_n$ in Corollary \ref{cor}, for $a=0$, $b=c=\alpha$. Hence, we have
\begin{align}
K_{d+1}(x)=&~2(x-\alpha)\alpha ^{d}T_d\left(\frac{x}{2\alpha} \right)-2\alpha ^{d+1}T_{d-1}\left(\frac{x}{2\alpha} \right)\nonumber\\
=&~2\alpha ^{d+1}\Bigg[T_{d+1}\left(\frac{x}{2\alpha} \right)-T_{d}\left(\frac{x}{2\alpha} \right)\Bigg],\nonumber
\end{align}
where $T_n(x)$ is the $n$-th Chebyshev polynomial of first kind and the above equalities follow from the properties of Chebyshev polynomials which can be found in \cite{oxford}.
The eigenvalues of $N_n$ are therefore the zeroes of the following equation:
$$\sin \left(\frac{2d+1}{2}\arccos \left(\frac{x}{2\alpha}\right)\right) \sin \left(\frac{1}{2}\arccos \left(\frac{x}{2\alpha}\right)\right)=0,$$
and the $x_j=p_1(j)$ are obtained.
To calculate the $p_i(j)$, $2\leq i\leq d$,
we must calculate the complex cofficient polynomial $\nu _i(x)$, where $x_i=\nu _i(x_1)$. Obviously, $\nu _1(x)=x$, and from (\ref{abc}), we get
$$
\begin{array}{c}
x_1x_1=2\alpha ^2+\alpha x_2 \Rightarrow
\nu _2(x)=\frac{1}{\alpha}\left(x^2-2\alpha ^2\right),\\
\\
x_1x_2=\alpha x_1+\alpha x_3 \Rightarrow
\nu _3(x)=\frac{1}{\alpha}\left(x\nu _2(x)-\alpha \nu _1(x)\right),\\
\vdots\\
x_1x_{d-1}=\alpha x_{d-2}+\alpha x_d \Rightarrow
\nu _d(x)=\frac{1}{\alpha}\left(x\nu _{d-1}(x)-\alpha \nu _{d-2}(x)\right).
\end{array}
$$
We now consider the recursive function $\varphi _n(x)$ in the form of:
$$\varphi _n(x) =x\varphi _{n-1}(x)-\alpha \varphi _{n-2},(x),$$
with $\varphi _1(x)=\alpha x$ and $\varphi _2(x)=x^2-2\alpha ^2$. From Lemma \ref{lem1}, $\varphi _n(x)$ is the following determinant
$$\left|
\begin{array}{cccccc}
\alpha x&1 &&&&\\
2\alpha ^2 &x/\alpha &1 &&&\\
&\alpha &x&1 &&\\
&&\ddots &\ddots & \ddots &\\
&&&\alpha &x &1\\
&&&&\alpha &x
\end{array} \right|_{n\times n}.$$
Laplace expansion yields
\begin{equation}\label{varp1}
\varphi _n(x)=(x^2-2\alpha ^2)H_{n-2}(x)-\alpha ^2xH_{n-3}(x),
\end{equation}
where $H_n(x)$ is the characteristic polynomial of the matrix:
$$\left(
\begin{array}{cccccc}
0&1 &&&&\\
\alpha  &0 &1 &&&\\
&\alpha &0&1 &&\\
&&\ddots &\ddots & \ddots &\\
&&&\alpha &0 &1\\
&&&&\alpha &0
\end{array} \right)_{n\times n},$$
and from  Lemma \ref{lem2}, we know that
$
H_n(x)=\left (\sqrt{\alpha}\right)^n U_n\left( \frac{x}{2\sqrt{\alpha}}\right).
$
Finally from (\ref{varp1}), we obtain
\begin{align}
\nu _i(x)=~&\frac{1}{\alpha} \varphi _i (x)\nonumber\\
=~&\left(\sqrt{\alpha}\right)^{i-4}\Bigg[(x^2-2\alpha ^2)U_{i-2}\left(\frac{x}{2\sqrt{\alpha}}\right)-
\alpha \sqrt{\alpha}x U_{i-3}\left(\frac{x}{2\sqrt{\alpha}}\right)
\Bigg],\nonumber
\end{align}
and from (\ref{character}), the proof is complete.
\end{proof}
\end{theorem}
We now calculate the characters of the P-polynomial table algebra whose first intersection matrix is given by (\ref{mat3}) in the following theorem.
 \begin{theorem}
Let $(A, \mathbf{B})$ be a P-polynomial table algebra
with $\mathbf{B}=\{x_0=1_A, x_1, \cdots , x_d\}$ and the first intersection matrix $B_1$ which is given in (\ref{mat3}).
Then the characters of $(A, \mathbf{B})$ are
$$
\begin{array}{l}
p_0(j)=1,\\
\\
p_1(j)=x_j=2\sqrt{\alpha \gamma}\cos\left( \frac{k\pi}{d}\right),\\
\\
p_i(j)=\frac{\left(\sqrt{\alpha}\right)^{i-2}}{\gamma}\Bigg[(x_j^2-2\alpha \gamma)U_{i-2}\left(\frac{x_j}{2\sqrt{\alpha}}\right)-\sqrt{
\alpha}\gamma x_j U_{i-3}\left(\frac{x_j}{2\sqrt{\alpha}}\right)
\Bigg], ~~~ 2\leq i\leq d,
\end{array}
$$
for $0\leq j \leq d$.
\begin{proof}
For each $i$, $0\leq i\leq d$, the $p _i(j)$, $0\leq j \leq d$, are equal to the eigenvalues of the $i$-th intersection matrix $B_i$.
Since $B_0=I_{d+1}$, we have $p _0(j) = 1$ for all $j$. Similarly, the $p _1(j)$ are equal to the eigenvalues of $B_1$.
Set $R_{d+1}(x)=|xI_{d+1}-B_1|$. Let $N_n$ be a tridiagonal matrix of the form
$$\left(
\begin{array}{cccccc}
0&\gamma &&&&\\
2\alpha  &0&\gamma &&&\\
&\alpha &0&\gamma &&\\
&&\ddots &\ddots & \ddots &\\
&&&\alpha &0 &\gamma \\
&&&&2\alpha &0
\end{array} \right)_{n\times n},$$
and $K_n(x)=|xI_{n}-N_n|$. From Lemma \ref{lem1}, we can obtain that $R_{d+1}(x)=K_{d+1}(x)$. Thus, we find the eigenvalues of $N_n$. By Laplace expansion, we get
$$K_{d+1}(x)=xC_{d}(x)-2\alpha \gamma C_{d-1}(x),$$
where $C_n(x)$ is the characteristic polynomial of $Q_n$ in Corollary \ref{cor}, for $a=0$, $b=\gamma$, $c=\alpha$. Hence, we have
\begin{align}
K_{d+1}(x)=&~2x(\sqrt{\alpha \gamma}) ^{d}T_d\left(\frac{x}{2\sqrt{\alpha\gamma}} \right)-4(\sqrt{\alpha \gamma}) ^{d+1}T_{d-1}\left(\frac{x}{2\sqrt{\alpha\gamma}} \right)\nonumber\\
=&~2(\sqrt{\alpha \gamma}) ^{d+1}\Bigg[T_{d+1}\left(\frac{x}{2\sqrt{\alpha \gamma}} \right)-T_{d-1}\left(\frac{x}{2\sqrt{\alpha \gamma}} \right)\Bigg]\nonumber,
\end{align}
where $T_n(x)$ is the $n$-th Chebyshev polynomial of the first kind. See  \cite{oxford}, for more details.
The eigenvalues of $B_1$ are therefore the zeroes of the following equation
$$\sin \left(d\arccos \left(\frac{x}{2\sqrt{\alpha\gamma}}\right)\right) \sin \left(\arccos \left(\frac{x}{2\sqrt{\alpha\gamma}}\right)\right)=0,$$
and the $x_j=p_1(j)$ are obtained. To calculate the $p_i(j)$, $2\leq i\leq d$,
we must calculate the complex coefficient polynomial $\nu _i(x)$, where $x_i=\nu _i(x_1)$. Obviously, $\nu _1(x)=x$, and from (\ref{abc}), we get
$$
\begin{array}{c}
x_1x_1=2\alpha \gamma +\gamma x_2 \Rightarrow
\nu _2(x)=\frac{1}{\gamma}\left(x^2-2\alpha \gamma \right),\\
\\
x_1x_2=\alpha x_1+\gamma x_3 \Rightarrow
\nu _3(x)=\frac{1}{\gamma}\left(x\nu _2(x)-\alpha \nu _1(x)\right),\\
\vdots\\
x_1x_{d-1}=\alpha x_{d-2}+\gamma x_d \Rightarrow
\nu _d(x)=\frac{1}{\gamma}\left(x\nu _{d-1}(x)-\alpha \nu _{d-2}(x)\right).
\end{array}
$$
We now consider the recursive function $\varphi _n(x)$ in the form of
$$\varphi _n(x) =x\varphi _{n-1}(x)-\alpha \varphi _{n-2},(x),$$
with $\varphi _1(x)=\gamma x$ and $\varphi _2(x)=x^2-2\alpha \gamma$. From Lemma \ref{lem1}, $\varphi _n(x)$ is the following determinant
$$\left|
\begin{array}{cccccc}
\gamma x&1 &&&&\\
2\alpha \gamma &x/\gamma &1 &&&\\
&\alpha &x&1 &&\\
&&\ddots &\ddots & \ddots &\\
&&&\alpha &x &1\\
&&&&\alpha &x
\end{array} \right|_{n\times n}.$$
Laplace expansion yields
\begin{equation}\label{varp2}
\varphi _n(x)=(x^2-2\alpha \gamma)H_{n-2}(x)-\alpha \gamma xH_{n-3}(x),
\end{equation}
where $H_n(x)$ is the characteristic polynomial of the matrix:
$$\left(
\begin{array}{cccccc}
0&1 &&&&\\
\alpha  &0 &1 &&&\\
&\alpha &0&1 &&\\
&&\ddots &\ddots & \ddots &\\
&&&\alpha &0 &1\\
&&&&\alpha &0
\end{array} \right)_{n\times n},$$
and from Lemma \ref{lem2}, we have
$
H_n(x)=\left (\sqrt{\alpha}\right)^n U_n\left( \frac{x}{2\sqrt{\alpha}}\right).
$
Finally, from (\ref{varp2}), we obtain
\begin{align}
\nu _i(x)=~&\frac{1}{\gamma} \varphi _i (x)\nonumber\\
=~&\frac{\left(\sqrt{\alpha}\right)^{i-2}}{\gamma}\Bigg[(x^2-2\alpha \gamma)U_{i-2}\left(\frac{x}{2\sqrt{\alpha}}\right)-\sqrt{
\alpha}\gamma x U_{i-3}\left(\frac{x}{2\sqrt{\alpha}}\right)
\Bigg],\nonumber
\end{align}
and from (\ref{character}), the proof is complete.

\end{proof}
\end{theorem}

\section{The LU Factors}\label{sec5}

Writing a given square matrix as the product of a lower triangular matrix and an upper triangular matrix is called LU-decomposition. Sometimes, an additional permutation (pivoting) is also required.

Let $M$ be a generic tridiagonal matrix of size $n$ as follows.

\begin{equation}\label{M}
M= \left(
\begin{array}{cccccc}
m_{1,1}&m_{1,2}&&&&\\
m_{2, 1}&m_{2, 2}&m_{2, 3}&&&\\
&m_{3, 2}&m_{3, 3}&m_{3, 4}&&\\
&&\ddots&\ddots & \ddots &\\
&&&m_{n-1, n-2} &m_{n-1, n-1} &m_{n-1, n}\\
&&&&m_{n, n-1}&m_{n, n}
\end{array} \right).
\end{equation}

The LU factorization of $M$ without pivoting has been exhaustively studied in~\cite{Bueno}. Here, we propose LU factors for $M$ with special pivoting in order to use them in the process of finding the characteristic equations of the tridiagonal matrices that we are considering in this work. We can also use these factors to efficiently solve systems of linear equations whose matrices are tridiagonal. This can be especially helpful when we use methods such as inverse power algorithms for finding the eigenvalues of such matrices (e.g.~see~\cite{AL}).

We consider specific pivoting to avoid division by $m_{i, i}$ for $i= 1, \cdots, n$ in the entries of $L$ and $U$. The following lemma summarizes the procedure for finding the LU decomposition of $M$.

\begin{lemma}\label{LU}
If $M$ is given by (\ref{M}) and $m_{i, i-1}\ne0$ for $i= 2, \cdots, n$, then $M=PLU$, where
\begin{equation}\label{P}
P= \left(
\begin{array}{ccccc}
0&0&\cdots&0&1\\
1&0&\cdots&0&0\\
0&1&\cdots&0&0\\
&&\ddots & &\\
0&0&\cdots&1 &0\\
\end{array} \right),
\end{equation}

\begin{equation}\label{L}
L= \left(
\begin{array}{cccccc}
1&&&&&\\
0&1&&&&\\
0&0&1&&&\\
&&& \ddots &&\\
0&0&0& \cdots&1&\\
L_{1, n}&L_{2, n}&L_{3, n}&\cdots&L_{n-1, n} &1\\
\end{array} \right),
\end{equation}
and

\begin{equation}\label{U}
U= \left(
\begin{array}{ccccccc}
m_{2,1}&m_{2,2}&m_{2,3}&0&0&\cdots&0\\
&m_{3,2}&m_{3,3}&m_{3,4}&0&\cdots&0\\
&&m_{4, 3}&m_{4, 4}&m_{4,5}&\cdots&0\\
&&&\ddots&\ddots&\ddots&\\
&&&&m_{n-1, n-2}&m_{n-1, n-1}&m_{n-1, n}\\
&&&&&m_{n, n-1}&m_{n, n}\\
&&&&&&U_{n, n}\\
\end{array} \right).
\end{equation}

The entries $L_{i, n}$ in (\ref{L}) satisfy the following recurrence relation for $3\leq i\leq n-1$ :

\begin{equation}\label{rec}
L_{i, n}= -\frac{1}{m_{i+1, i}}(m_{i, i}L_{i-1, n}+ m_{i-1, i}L_{i-2, n}),
\end{equation}
with $L_{1, n}= \frac{m_{1,1}}{m_{2, 1}},$ and $L_{2, n}= -{\frac {m_{{1,1}}m_{{2,2}}-m_{{1,2}}m_{{2,1}}}{m_{{2,1}}m_{{3,2}}}}.$ Moreover,

\begin{equation}\label{Ue}
U_{n,n}= -m_{n, n}L_{n-1, n}- m_{n-1, n}L_{n-2, n}.
\end{equation}
\end{lemma}

\begin{proof}
The proof is straightforward and can be fairly easily done by induction.
\end{proof}

We can now write $\det{M}= \det{P}\times\det{L}\times\det{U}$. From (\ref{P}), (\ref{L}), and (\ref{U}), we conclude that $\displaystyle{\det{P}= (-1)^{n-1}}$, $\det{L}= 1$, and $\displaystyle{\det{U}= \prod_{k= 2}^{n}{m_{k, k-1}}U_{n, n}}$, respectively. From (\ref{Ue}), we get
\begin{equation}\label{detM}
\displaystyle{\det{M}=(-1)^{n}(\prod_{k= 2}^{n}{m_{k, k-1}})(m_{n, n}L_{n-1, n}+ m_{n-1, n}L_{n-2, n})}.
\end{equation}

\subsection{Eigenstructure of $B_1$}
Finding the eigenstructure of $B_1$ from (\ref{matrix}) is an important case study of the applications of Lemma~\ref{LU} and its results. We let $\Lambda_2= xI_{d+ 1}- B_1$ to get
\begin{equation}\label{lambda2}
\Lambda_2= \left(
\begin{array}{ccccc}
x&-1&&&\\
-k&x-a_1&-c_2&&\\
&-b_1&x-a_2&\ddots &\\
&&\ddots & \ddots &-c_{_{d}}\\
&& &-b_{_{d-1}} &x-a_{_d}\\
\end{array} \right).
\end{equation}
Now from Lemma~\ref{LU} and (\ref{detM}), it turns out that
\begin{equation}\label{char1}
\det{\Lambda_2}= -k(\prod_{j=1}^{d}b_j)((x-a_d)L_{d, d+1}- c_dL_{d-1, d+1}),
\end{equation}
where the $L_{i, d+1}$ for $i= 3, \cdots d$ is the solution of

\begin{equation}
L_{i, d+1}= \frac{1}{b_{i-2}}((x-a_{i-1})L_{i-1, d+1}- c_{i-1} L_{i-2, d+1}),
\end{equation}

\noindent with $$L_{1, d+1}= -\frac{x}{k},$$ and $$L_{2, d+1}= -{\frac {x(x-a_1)- k}{kb_1}}.$$
The characteristic polynomial of $B_1$ is therefore given by $$(x-a_d)L_{d, d+1}- c_dL_{d-1, d+1},$$ from which we can find its eigenvalues.

The right and left eigenvectors of $B_1$ can also be found. Again, the results are stated below as a theorem.

\begin{theorem}
If $x$ is an eigenvalue of $B_1$, then the entries of $\mathbf{\eta}$, the right eigenvector of $B_1$ associated with $x$, for $j= d+1, d, \cdots, 4$, satisfy the following recurrence relation:
\begin{equation}
\eta_{j-2}= \frac{(x-a_{j-2})}{b_{j-3}}\eta_{j-1}- \frac{c_{j-1}}{b_{j-3}}\eta_j,
\end{equation}

\noindent with arbitrary $\eta_{d+1}(\ne0)$ and $\eta_{d}= \frac{(x-a_d)}{b_{d-1}}\eta_{d+1}$. When $j= 3$, $$\eta_{1}= \frac{(x-a_1)}{k}\eta_{2}- \frac{c_2}{k}\eta_3.$$ Moreover, the entries of $\mathbf{\psi}$, the left eigenvector of $B_1$ associated with $x$, for $4\leq j\leq d+1$, satisfy the following recurrence relation:
\begin{equation}
\psi_{j}= \frac{(x-a_{j-2})}{b_{j-2}}\psi_{k-1}- \frac{c_{j-2}}{b_{j-2}}\psi_{k-2},
\end{equation}

\noindent with arbitrary $\psi_1(\ne0)$ and $\psi_2= \frac{(x)}{k}\psi_1$. When $j= 3$, $$\psi_{3}= \frac{(x-a_1)}{b_1}\psi_{2}- \frac{1}{b_1}\psi_{1}.$$
\end{theorem}

One can use this eigenstructure $B_1$ to find the eigenvalues and eigenvectors of them through methods such as Constrained Rayleigh Quotient (see e.g.~\cite{AL} for more details).

\section{Concluding remarks}\label{con}

The focus of our work in this paper is on the application of tridiagonal matrices in P-polynomial table algebras. More precisely, the first intersection matrix of a P-polynomial table algebra is tridiagonal whose eigenvalues can be used in calculating the characters of the table algebra. We have studied the eigenstructure of some tridiagonal matrices. We have also used our results about tridiagonal matrices in  studying the characters of two classes of P-polynomial table algebras.
 Moreover, we have calculated the LU factors for $B_1$, the latter being the intersection matrix of any  P-polynomial table algebra.
Finally, the results from Section~\ref{sec5} can be used in devising numerical algorithms for finding the eigenvalues and eigenvectors of tridiagonal matrices in general, and specifically those that we are more interested in here.

\bibliographystyle{amsplain}

\end{document}